\documentclass{amsart}

\usepackage{amssymb,color}
\usepackage{amsfonts}
\usepackage{amsmath}
\usepackage{euscript}
\usepackage{enumerate}
\usepackage{pdfsync}
\newtheorem{theorem}{Theorem}[section]
\newtheorem{lemma}[theorem]{Lemma}

\newtheorem{prop}[theorem]{Proposition}
\newtheorem{cor}[theorem]{Corollary}
\newtheorem{exa}[theorem]{Example}

\newtheorem*{Theorem1'}{Theorem 1'}

\theoremstyle{definition}

\theoremstyle{remark}

\def\Z{{\bf Z}}

\def\N{{\bf N}}

\begin{document}

\title[Substitution groups of formal power series]{Substitution groups of formal power series}

\author{Agust\'{i}n D'Alessandro}
\address{Department of Mathematics and Statistics, University of Regina, Canada}
\email{dalessandro.ag@gmail.com}

\author{Fernando Szechtman}
\address{Department of Mathematics and Statistics, University of Regina, Canada}
\email{fernando.szechtman@gmail.com}
\thanks{The second author was supported in part by NSERC discovery grant 2020-04062}


%\date{January 1, 2001 and, in revised form, June 22, 2001.}



\begin{abstract} Let $G$ be the group of power series $x+a_2x^2+a_3x^3+\cdots\in R[[x]]$ under substitution,
where $R$ is a commutative ring  with $1\neq 0$ of prime characteristic $p$. Given any $n\geq 1$,
the subgroup $K_n=\{x+a_{n+1}x^{n+1}+a_{n+2}x^{n+2}+\cdots\,|\,  a_i\in R\}$ is normal in $G$,
and the quotient $G_n=G/K_n$ is the group of truncated polynomials over $R$ of degree $\leq n$ under substitution.
In this paper, we compute the exponent of the image of $K_r$ in $G_n$, for all $r,n\geq 1$, indicating in every case a family of elements realizing this exponent.
\end{abstract}

\maketitle

\section{Introduction} 

Let $G$ be the group of power series $x+a_2x^2+a_3x^3+\cdots\in R[[x]]$ under substitution,
where $R$ is a commutative ring  with $1\neq 0$, so that
$f*g=f(g)$ for $f,g\in G$. The first systematic study of $G$ was made
by Jennings \cite{J} in 1954, followed three decades later by Johnson \cite{Jo} and York \cite{Y,Y2}.
Ever since then, $G$ has received and continues to attract considerable attention
due to its remarkable properties and the important role it plays in the theory of pro-$p$-groups when
$R$ is a finite field, in which case $G$ is usually known as the Nottingham group. %One of the striking features
%of~$G$ when $R=\Z_p$, $p$ a prime, is that it is universal, in the sense that it contains a copy of every  
%pro-$p$-group having a countable basis. This was shown by Camina \cite{C}.
We refer to the survey articles by Camina \cite{C} and Babenko \cite{B}, and references therein, for a wealth of 
%other 
properties enjoyed by $G$.

%, extending a prior result unpublished by Leedham-Green and Weiss (based on work by Witt on \cite{?} on Galois theory),

We assume for the remainder of the paper that $R$ has prime characteristic $p$. 
Given any $n\geq 1$, we consider the normal subgroup $K_n=\{x+a_{n+1}x^{n+1}+a_{n+2}x^{n+2}+\cdots\,|\,  a_i\in R\}$ of $G$ and
the quotient $G_n=G/K_n=\{x+a_{2}x^2+\cdots+a_n x^n\,|\,  a_i\in R\}$, the group of truncated polynomials over~$R$ of degree 
$\leq n$ under substitution.
In 1990, York~\cite{Y} computed the exponent of $G_n$ when $p>2$, as well as when $p=2$ provided
$R$ is an integral domain,
exhibiting as well an element
realizing this exponent. According to Babenko \cite[Section 4.2]{B}, finding the exponent
of $G_n$ when $R=\Z_p$ is an important and natural problem, and its exact calculation is
a delicate task when $R$ is arbitrary. Writing $\pi_n:G\to G_n$ for the canonical projection,
we set  $H_r=\pi_n(K_r)$, $r\geq 1$. This is trivial
if $n\leq r$, and the subgroup $H_r=\{x+a_{r+1}x^{r+1}+\cdots+a_n x^n\,|\,  a_i\in R\}$ of $G_n$ if $r<n$.
In this paper, we extend York's result by computing the exponent
of $H_r$ for arbitrary $r,n\geq 1$, indicating in every case a family of elements realizing this 
exponent. Our study includes, in particular, the case $r=1$ and $p=2$ left open by York, when, perhaps surprisingly, 
the answer ultimately depends on whether $R$ is a Boolean ring or not. As a byproduct of our computations,
we lay the foundations for the determination of the power-commutator presentation (PCP) of $G_n$, when $R=\Z_p$,
relative to its standard generators, namely $\pi_n(x+x^2),\dots,\pi_n(x+x^n)$. 
This, in turn, will be used to study the automorphism group of $G_n$, taking advantage
of the fact that the PCP of $G_n$ yields that of $G_{n-1}$ (after removing the last generator of $G_n$
and all defining relations involving it) and appealing to the canonical projection $G_n\to G_{n-1}$, $n>1$,
to derive information about $\mathrm{Aut}(G_n)$ from that of $\mathrm{Aut}(G_{n-1})$. For $R=\Z_p$, $p\geq 5$, 
all automorphisms of $G$ are induced by an inner automorphism of $\Z_p^\times\ltimes G$ (where $\Z_p^\times$ is viewed
as the group of all $ax$, $a\in \Z_p^\times$, acting on $G$ by conjugation), as shown by Klopsch~\cite{K}. Nothing like this
holds for the finite $p$-groups $G_n$, {\color{red} $n>2$}, which admit more symmetry.

For $r\geq 1$, let $t$ be the remainder of dividing $r$ by $p$. In \cite[Section 4]{C}, we find that
$K_r^p\subseteq K_{pr+t}$ when $R$ is a finite field. This easily extends to a general ring $R$,
as found in Theorem \ref{kr} below, and gives the upper bound $p^m$ for the exponent of $H_r$ 
whenever $n\leq p^m r+p^{m-1}t+\cdots+pt+t$. Much of the paper is devoted to showing that this
bound is exact if we also have $$p^{m-1} r+p^{m-2}t+\cdots+pt+t<n,$$ except when $r=1$, $p=2$, $n=2^m$, $m>2$,
and $R$ is a Boolean ring. 

If $R$ is a finite field and $p$ is odd, \cite[Theorem 6]{C} proves the stronger result,
attributed to Leedham-Green and McKay,  that $K_r^p=K_{pr+t}$.
Given a finite $p$-group $T$, {\color{red} if $m\geq 1$, we write $T^m$ for the subgroup of $T$ generated by all $m$th powers
of elements of $T$, and recursively define $T(1)=T^p$, $T(m)=T(m-1)^p$ for $m>1$.}
In general, $T^{p^m}$ is properly included in~$T(m)$. If $T$ is regular or powerful, then equality holds,
as shown in \cite[Sections 1.2 and 6.1]{LM}. As $H_r$ is, in general, neither regular nor powerful,
(see Example \ref{ejemplo} below), it is not possible to deduce solely from $K_r^p=K_{pr+t}$ that the exponent of $H_r$ is
as high as $p^m$ when $R$ is a finite field, $p$ is odd, and 
$p^{m-1} r+p^{m-2}t+\cdots+pt+t<n\leq p^m r+p^{m-1}t+\cdots+pt+t$. The calculation of the exponent of $H_r$
in this case, and more generally when $R$ is arbitrary, requires considerable effort.

This paper can be viewed as a contribution to the study of the power structure of $G$.
In this regard, a notoriously difficult problem is to find explicit elements of order $p^n$ in $G$.
A recent paper by Byszewski, Cornelissen, and Tijsma \cite{BCT} uses finite automata 
to characterize the elements of order $p^n$ (in terms of their sequence of coefficients) for various values of $p^n$.
We also refer to \cite{BCT} for a careful review of prior work in this area.

Following York~\cite{Y}, we write $f^{(m)}$ for the $m$th iterate of 
$f\in G$ to distinguish it from its $m$th power $f^m$ in $R[[x]]$.
Given any $a,b\in R$ and $f=x+a x^{r+1}+b x^{r+t+1}+\cdots\in K_r$, we demonstrate that
$f^{(p)}=x+A x^{pr+t+1}+B x^{pr+2t+1}+\cdots$,
where $A$ and $B$ are explicitly determined in terms of $a,b,r,t$, and~$p$, and all powers of $x$
between $x^{pr+t+1}$ and $x^{pr+2t+1}$ are shown to have a zero coefficient (besides those between  $x$ and $x^{pr+t+1}$).
This nontrivial result is divided into several cases, depending on whether $p$ is odd or even,
and whether $t\equiv 0,-1\mod p$ or not. The most challenging one occurs when $p>2$
and $t\not\equiv 0,-1\mod p$, dealt with in Theorem \ref{improve2}, reads as follows.

\medskip

\noindent{\bf Theorem A. }{\it Suppose that $r\not\equiv 0,-1\mod p$ and let
$
f=x+a x^{r+1}+b x^{r+t+1}+\cdots.
$
Then
$$
f^{(p)}=x+A x^{pr+t+1}+B x^{pr+2t+1}+\cdots,
$$
where if $r>t$,
$$
A=-a^{p-1}b\quad\text{ and }\quad B=-a^{p-2}b^2,
$$
while if $r=t$, 
$$
A=-a^{p-1}b+\frac{r+1}{2}a^{p+1}\quad\text{ and }\quad B=-a^{p-2}b^2+{\color{red}(r+1)a^{p}b}-\frac{(r+1)^2}{4}a^{p+2}.
$$
}

Writing $s=pr+t$, we see that $f^{(p)}\in K_s$ has the same shape as $f$ and that $s\equiv t\mod p$.
Our calculation of the exponent of $H_r$ follows by repeatedly applying Theorem A after suitable initial choices for $a$ and $b$.
Some of the consequences of Theorem A are used elsewhere \cite{S} in the determination
of the PCP of $G_n$, $n\leq 15$, relative to its standard generators when $R=\Z_p$.

Of the various cases in which the analysis of the exponent of $H_r$ must be divided, we single out the following.

\medskip

\noindent{\bf Theorem B. }{\it Suppose $p=2$. If $R$ is a Boolean ring
the exponent of $G_{2^m}$ is $2^{m-1}$ for every $m>2$.
If $R$ is not a Boolean ring, the exponent of $G_{2^m}$ is $2^{m}$ for every $m>2$.
}

\medskip

\noindent{\bf Theorem C. }{\it Suppose that $p$ is odd, that $0\leq t<p$, 
and that for some $m,n\in\N$, we have
$$p^{m-1} r+p^{m-2}t+\cdots+pt+t<n\leq p^m r+p^{m-1}t+\cdots+pt+t.$$ Then the exponent of $H_r$ is exactly $p^m$.
}

%The present paper aims to strengthen properties of $G(R)$ first found by Jennings and later extended by York.

\section{Notation} 

York \cite[p. 484]{Y} associates an upper triangular infinite matrix $M$ to any $f\in G$, where $M_{i,j}$ is the coefficient of $x^j$
in $f^i$ for all $1\leq i<j$. If the matrix associated to $g\in G$ is $N$, then
the matrix associated to $f*g$ is $MN$. He also defines the matrix $\Delta$ by $M=I+\Delta$, noting that $M^p=I+\Delta^p$,
so that $f^{(p)}$ can be read off from the first row of $\Delta^p$. Given any $d>1$, we let $L(d)$ stand
for the set of
all sequences $j=(j_0,j_1,\dots,j_p)$ such that $1=j_0<j_1<\cdots<j_p=d$, and set $S(j)=\Delta_{j_0,j_1}\cdots \Delta_{j_{p-1},j_p}$. We write $J(d)$ for the possibly empty set of all $j\in L(d)$ such that $S(j)\neq 0$. With this notation, we have
$$
\Delta^p_{1,d}=\underset{j\in J(d)}\sum S(j),
$$
a formula to be used repeatedly to study $f^{(p)}$. Given $j\in L(d)$, we refer to the differences $j_{i+1}-j_i$, where $0\leq i<p$, as {\em jumps}. 

We will let $r$ and $t$ stand for integers satisfying $r>0$ and $t\geq 0$,
and $a$ and $b$ for arbitrary elements of $R$.

\section{Background}

The next three results hold in arbitrary characteristic.

\begin{lemma}\label{cong} Given $n\geq 0$, $f,h\in R[[x]]$, and $g,k\in  G(R)$ satisfying 
$f\equiv h\mod x^{n+1}$ and $g\equiv k\mod x^{n+1}$,
we have
$$
f(g)\equiv h(k)\mod x^{n+1}.
$$
%This ensures that the coefficients of $x,\dots,x^n$ appearing in $f(g)$ and $h(k)$ are identical.
\end{lemma}

\begin{proof} Clearly,
$$
f(g)\equiv f_0+f_1g+f_2g^2+\dots +f_ng^n\mod x^{n+1},\; 
h(k)\equiv h_0+h_1k+h_2k^2+\dots +h_nk^n\mod x^{n+1}.
$$
As $f\equiv h\mod x^{n+1}$, {\color{red} we deduce}
$
f_0=h_0,\dots, f_n=h_n.
$
Moreover, from $g\equiv k\mod x^{n+1}$, we infer 
$
g\equiv k,\dots, g^n\equiv k^n\mod x^{n+1},
$
so $f(g)\equiv h(k)\mod x^{n+1}$.
\end{proof}

\begin{lemma}\label{tail} Let $f=x+f_2x^2+\cdots\in G$, $n>1$, and $a\in R$. Then
$$
f+ax^n\equiv f * (x+ax^n)\mod x^{n+1},
$$
so  that the inverse {\color{red}$f^{(-1)}$} of $f$ in $G$ satisfies
$$
{\color{red}f^{(-1)}} * (f+a x^n)\equiv x+a x^n\mod x^{n+1}.
$$
\end{lemma}

\begin{proof} {\color{red} First of all, we have}
$$
f * (x+ax^n)\equiv x+ax^n+f_2(x+ax^n)^2+f_3(x+ax^n)^3+\cdots\equiv f+ax^n\mod x^{n+1}.
$$
{\color{red} It now follows from Proposition \ref{cong} that
$$
f^{(-1)} * (f+a x^n)\equiv f^{(-1)} * (f * (x+ax^n))\equiv (f^{(-1)} * f) * (x+ax^n)\equiv x+ax^n\mod x^{n+1}.
$$
}
\end{proof}

\begin{prop}(Jennings \cite[Section 2]{J})\label{com1} If $r,s>1$ and 
$$
f=x+f_rx^r+\cdots\in K_{r-1},\; g=x+g_sx^s+\cdots \in K_{s-1},
$$
then {\color{red} $[f,g]=f^{(-1)}*g^{(-1)}*f*g$ satisfies}
$$
[f,g]=x+(r-s)f_r g_s x^{r+s-1}+\cdots\in K_{(r-1)+(s-1)}.
$$
\end{prop}

\begin{proof} We have
$$
f*g\equiv x+g_s x^{s}+\cdots+g_{r+s-1} x^{r+s-1}+f_r x^{r}+\cdots+f_{r+s-1} x^{r+s-1}+r f_r g_s x^{r+s-1}\mod x^{r+s},
$$
and likewise
$$
g*f\equiv x+f_r x^{r}+\cdots+f_{r+s-1} x^{r+s-1}+g_s x^{s}+\cdots+g_{r+s-1} x^{r+s-1}+ s g_s f_r x^{r+s-1}\mod x^{r+s}.
$$
Thus
$$
f*g\equiv g*f+(r-s)f_r g_s x^{r+s-1}\mod x^{r+s}.
$$
Applying first Lemma \ref{tail} and then Lemma \ref{cong} yields the desired formula.
\end{proof}

\begin{lemma}\label{diryo} Let $f=x+a x^{r+1}+\cdots\in K_r$ and suppose that $\Delta_{u,v}\neq 0$
for some $u<v$. Then $v-u\geq r$. Thus, if $j\in J(d)$, then $j_{i+1}-j_i\geq r$ for all $0\leq i<p$.
\end{lemma}

\begin{proof} This is clear.
\end{proof}

\begin{theorem}\label{je} (Jennings \cite[Theorem 2.5]{J}) The following inclusion always holds: $K_r^p\subseteq K_{pr}$.
\end{theorem}

\begin{proof} Let $f\in K_r$.  We  wish to show that $f^{(p)}\in K_{pr}$. Since $f^{(p)}$
is represented by $M^p=I+\Delta^p$, this translates
as follows: if $d>1$ and $\Delta^p_{1,d}\neq 0$, then $d>pr$.
As $f\in K_r$, successive application of Lemma \ref{diryo} yields $j_1\geq r+1,\dots,j_p\geq pr+1$, as needed.
\end{proof}

If $R$ is a finite field, $p$ is odd, and $t$ is the remainder of dividing $r$ by $p$,
then \cite[Theorem 6]{C} shows that $K_r^p=K_{pr+t}$. For our purposes, we need the following result.

\begin{theorem}\label{kr} Let $t$ be the remainder of dividing $r$ by $p$. Then
$K_r^p\subseteq K_{pr+t}$.
In particular, for any $m,n\in\N$ such that $n\leq p^m r+p^{m-1}t+\cdots+pt+t$, the exponent of $K_r$ in
$G_n$ is at most $p^m$.
\end{theorem}

\begin{proof} Let $f\in K_r$. If $f^{(p)}$ is trivial there is nothing to do, so we may assume that neither $f$ nor $f^{(p)}$
are trivial. Thus $f=x+a x^{s+1}+\cdots$, where $s\geq r$ and $a\neq 0$. If $s>r$, then Theorem \ref{je} yields 
$f^{(p)}\in K_{ps}\subset K_{pr+t}$. Suppose next that $s=r$ and $R$ is an integral domain.
By Theorem~\ref{je}, $f^{(p)}=x+b x^{pr+u+1}+\cdots$, where $u\geq 0$
and $b\neq 0$. Since $f$ and $f^{(p)}$
commute with each other, Proposition \ref{com1} yields $r\equiv u\mod p$, so $u\geq t$.
In general, there is an integral domain $S$ of characteristic $p$ and a ring epimorphism $S\to R$, which yields a group
epimorphism $G(S)\to G(R)$, so the result for $G(R)$ follows from that for $G(S)$.
\end{proof}

\begin{lemma} (York \cite[p. 486]{Y}) \label{York} If $f\in G$, $u<v$, $p\mid u$, and $p\nmid v$, then $\Delta_{u,v}=0$.
\end{lemma}

\begin{proof} Writing $u = pw$, we have
$$
f^u=f^{pw}=((x+f_2 x^2+f_3 x^3+\cdots)^p)^w=(x^p+f_2^p x^{2p}+f_3^p x^{3p}+\cdots)^w,
$$
so the only powers of $x$ appearing in $f^u$ must be multiples of $p$.
\end{proof}

\section{The $p$th iterate of certain elements of $K_r$}\label{s3}

\begin{lemma}\label{dirto} If $r\geq t$ and
$f=x+a x^{r+1}+b x^{r+t+1}+\cdots\in K_r$,
then for any $j\in J(d)$, we have
\begin{equation}\label{dirt}
j_{i+1}-j_i=r,\text{ or }j_{i+1}-j_i=r+t,\text{ or }j_{i+1}-j_i>r+t,\quad 0\leq i<p.
\end{equation}
\end{lemma}

\begin{proof} This is clear.
\end{proof}

\begin{prop}\label{sub1} If $r\equiv 0\mod p$ and 
$
f=x+a x^{r+1}+\cdots,
$
then 
$
f^{(p)}=x+a^p x^{pr+1}+\cdots,
$
and therefore for any $m\in\N$, we have
$
f^{(p^m)}=x+a^{p^m} x^{p^m r+1}+\cdots.
$
\end{prop}

\begin{proof} The coefficient of $x^{pr+1}$ in $f^{(p)}$ is $\Delta_{1,r+1}\Delta_{r+1,2r+1}\cdots\Delta_{(p-1)r+1,pr+1}=a^p$.
\end{proof}

For the remainder of the paper we assume that $r\equiv t\mod p$ and $r\geq t$, though we do not necessarily require that 
$0\leq t<p$. This restriction will play a role in Section \ref{s5}.

\begin{prop}\label{sub2} Suppose that $p$ is odd, $r\equiv -1\mod p$,
and 
$$
f=x+a x^{r+1}+bx^{r+t+1}+f_{r+t+2}x^{r+t+2}+f_{r+t+3}x^{r+t+3}+\cdots.
$$
Then
$
f^{(p)}=x-a^{p-1}b x^{pr+t+1}-a^{p-2}b^2 x^{pr+2t+1}+\cdots
$
if $t=p-1$, while if $t=up+(p-1)$ for some $u\geq 1$, then
$$
f^{(p)}=x-a^{p-1}b x^{pr+t+1}-a^{p-1}(f_{r+t+1+p} x^{pr+t+1+p}+\cdots+f_{r+t+1+up} x^{pr+t+1+pu})-a^{p-2}b^2 x^{pr+2t+1}+\cdots
$$
In particular, if $t=p-1$, $f=x+x^{r+1}+x^{r+t+1}+\cdots$, and $m\geq 1$, then
$$
f^{(p^m)}=x+(-1)^m x^{p^m r+p^{m-1}t+\cdots+pt+t+1}+(-1)^{m} x^{p^m r+p^{m-1}t+\cdots+pt+2t+1}+\dots.
$$
\end{prop}

\begin{proof} Take any $1<d\leq pr+2t+1$ such that $\Delta^p_{1,d}\neq 0$ and any $j\in J(d)$.
We claim that either $j_1=r+t+1$ or else $t>p$ and $j_1=r+t+1+vp$ for some $1\leq v\leq u$. Three cases arise:

$\bullet$ $j_1=r+1$, which is congruent to 0 modulo $p$. Then $p\mid j_i$ for all $1\leq i\leq p$ by Lemma \ref{York}.
But $r\equiv -1\mod p$ and $r+t\equiv -2\not\equiv 0\mod p$, so $j_{i+1}-j_i\geq r+t+2$ for all $1\leq i<p$ by (\ref{dirt}),
hence {\color{red} $j_p\geq r+1+(p-1)(r+t+2)=pr+(p-1)t+2p-1>pr+2t+1$}, a contradiction.

$\bullet$ $j_1>r+2t+1$. Then $j_p>pr+2t+1$, which is absurd.

$\bullet$ $r+t+1<j_1\leq r+2t+1$. Then $j_1+(p-2)r+(r+t)>pr+2t+1$, so for $j_p\leq pr+2t+1$ we require
$j_{i+1}-j_i=r$ for all $1\leq i<p$ by (\ref{dirt}). Here $j_1=r+t+1+k$ for some $1\leq k\leq t$.
Let $h$ be remainder of dividing $k$ by $p$. If $h\neq 0$, then $j_h=hr+t+1+k$ is congruent to 0 modulo $p$,
and $j_{h+1}=j_h+r$ is not congruent to 0 modulo~$p$, against Lemma \ref{York}. If $h=0$, then 
$k=pv$ for some $1\leq v\leq u$, so $t>p$.

This proves the claim. It now follows from (\ref{dirt}) that if $j_1=r+t+1$, then $d=pr+t+1$ and
$j$ consists of $p-1$ jumps of length $r$ plus one jump of length $r+t$ or else $d=pr+2t+1$ and 
$j$ consists of $p-2$ jumps of length $r$
plus two jumps of length $r+t$, while if $j_1=r+t+1+pv$ for some $1\leq v\leq u$, then $t>p$, $d=pr+t+1+pv$,
and $j$ consists of $p-1$ jumps of length $r$
plus one jump of length $r+t+pv$. We proceed to compute the value of $\Delta^p_{1,d}$ in each case.

Assume first that $d=pr+t+1$.  Then
$$
\Delta^p_{1,pr+t+1}=\Delta_{1,r+t+1}\Delta_{r+t+1,2r+t+1}\cdots\Delta_{(p-1)r+t+1,pr+t+1}=b(r+t+1)a\cdots((p-1)r+t+1)a.
$$
As $r\not\equiv 0\mod p$, this is equal to $-a^{p-1}b$ by Wilson's theorem. The case when $d=pr+t+1+pv$ is entirely
analogous. Assume finally that $d=pr+2t+1$. Then the second jump of length $r+t$ must occur last if it is to produce a nonzero 
summand $\Delta_{j_{0},j_1}\cdots\Delta_{j_{p-1},j_p}$, for otherwise
$j_{p-1}=(p-1)r+2t+1$ is congruent to 0 modulo $p$, and $j_p\not\equiv 0\mod p$, forcing $\Delta_{j_{p-1},j_p}=0$
by Lemma \ref{York}. Thus
$$
\Delta^p_{1,pr+2t+1}=\Delta_{1,r+t+1}\Delta_{r+t+1,2r+t+1}\cdots\Delta_{(p-2)r+t+1,(p-1)r+t+1}\Delta_{(p-1)r+t+1,pr+2t+1}.
$$
As $r\equiv -1\mod p$, the product of the first $p-1$ factors is equal to
$$b(r+t+1)a\cdots((p-2)r+t+1)a=-a^{p-2}b.$$ The last
factor is the coefficient of $x^{pr+2t+1}$ in $f^{(p-1)r+t+1}$. If $r>t$, the only way to produce $x^{pr+2t+1}$
is by selecting $x$ exactly $(p-1)r+t$ times and $x^{r+t+1}$ once, which gives
$$
\Delta_{(p-1)r+t+1,pr+2t+1}=((p-1)r+t+1)b=b,
$$
confirming the stated value of $\Delta^p_{1,pr+2t+1}$. If $r=t$, there is a second way to produce $x^{pr+2t+1}$, namely
by selecting $x$ exactly $(p-1)r+t-1$ times and $x^{r+1}$ twice, which gives
$$
\Delta_{(p-1)r+t+1,pr+2t+1}=((p-1)r+t+1)b+{\binom{(p-1)r+t+1}{2}}a^2=b,
$$
since $(p-1)r+t+1\equiv 1\mod p$. This gives the desired value of $\Delta^p_{1,pr+2t+1}$ in this case as well.

The last formula $(x+x^{r+1}+x^{r+t+1}+\cdots)^{(p^m)}$ now follows by repeated application of the previous one,
applied with $t=p-1$ and $a=b=1$.
\end{proof}

\begin{prop}\label{sub3} Suppose that $f=x+a x^{r+1}+b x^{r+t+1}+\cdots$ as well as $r\not\equiv 0,-1\mod p$.
Then
$
f^{(p)}=x+A x^{pr+t+1}+\cdots,
$
where
$$
A=\begin{cases} -a^{p-1}b & \text{ if }r>t,\\ -a^{p-1}b+\frac{r+1}{2}a^{p+1}& \text{ if }r=t.\end{cases}
$$
\end{prop}

\begin{proof} We claim that if $\Delta_{1,d}^p\neq 0$ for some $d>1$,  then $d\geq pr+t+1$.
Indeed, let $j\in J(d)$. If $j_1\geq r+t+1$, then $j_p\geq pr+t+1$ by Lemma \ref{diryo}.
If $j_1<r+t+1$, then necessarily $j_1=r+1$. If $j$ contains at least one jump of length $>r$,
then Lemma \ref{dirto} implies $j_p\geq pr+t+1$. If $j$ consists of $p$ jumps of length $r$,
then one of the numbers $r+1,\dots,(p-1)r+1$ must be congruent to 0 modulo $p$, while $j_p=pr+1$
is not, against Lemma \ref{York}. 

This proves the claim, so $f^{(p)}\in K_{pr+t}$, and we proceed to determine $A$.
Set $d=pr+t+1$ and suppose $j\in J(d)$. By (\ref{dirt}), of the $p$ differences $j_{i+1}-j_i$, $0\leq i<p$,
one is equal to $r+t$ and the remaining $p-1$ are equal to $r$. As $r\not\equiv 0,-1\mod p$, there is one and
only one~$k$ such that $1<k\leq p-1$ and $kr+1\equiv 0\mod p$. Suppose, if possible, that $j_{i+1}-j_i=r+t$
for some $0\leq i<k-1$. Then {\color{red} $j_{k-1}=(k-1)r+t+1$, which is congruent to 0 modulo $p$, and since $j_{k}=j_{k-1}+r$,}
we have $\Delta_{j_{k-1},j_k}=0$ by Lemma \ref{York}, a contradiction. Then
$$
j_1=r+1,\dots, j_{k-1}=(k-1)r+1.
$$
The case $j_k=kr+1$ is impossible by Lemma \ref{York}, as $pr+t+1\not\equiv 0\mod p$. Thus 
$$
j_k=kr+t+1,\dots, j_{p}=pr+t+1,
$$
and therefore
$$
A=\Delta_{1,r+1}\cdots\Delta_{(k-1)r+1,kr+t+1}\cdots\Delta_{(p-1)r+t+1,pr+t+1}.
$$
Here
$$
\Delta_{j_i,j_{i+1}}=\begin{cases} (ir+1)a & \text{ if }i<k-1,\\ ((i+1)r+1)a & \text{ if }i>k-1.\end{cases}
$$
On the other hand,
$$
\Delta_{j_{k-1},j_{k}}=\begin{cases} ((k-1)r+1)b & \text{ if }r>t,\\ ((k-1)r+1)b+{\binom{(k-1)r+1}{2}}a^2  & \text{ if }r=t.\end{cases}
$$
The stated value of $A$ now follows from Wilson's theorem.
\end{proof}

\begin{prop}\label{sub5} Suppose that $p=2$, $r\equiv 1\mod 2$, and that
$
f=x+a x^{r+1}+b x^{r+2}+\cdots.
$ 
Then 
$
f^{(2)}=x+ A x^{2r+2}+B x^{2r+3}+\cdots,
$
where
$$
A=\begin{cases} ab & \text{ if }r>1,\\ a(a^2+b) & \text{ if }r=1,
\end{cases}\; 
B=\begin{cases} b^2 & \text{ if }r>1,\\ b(a^2+b) & \text{ if }r=1.
\end{cases}
$$
\end{prop}

\begin{proof} The coefficient of $x^{2r+2}$ in $f^{(2)}$ is 
$\Delta_{1,r+2}\Delta_{r+2,2r+2}=ab$ if $r>1$ and $\Delta_{1,r+2}\Delta_{r+2,2r+2}+\Delta_{1,r+1}\Delta_{r+1,2r+2}=ab+a^3$
if $r=1$. Moreover, the coefficient of $x^{2r+3}$ in $f^{(2)}$ is 
$\Delta_{1,r+2}\Delta_{r+2,2r+3}$, which equals $b^2$ if $r>1$ and $b^2+a^2b$
if $r=1$.
\end{proof}

\section{The exponent of $H_r$ when $r\equiv 0,-1\mod p$}\label{s4}

\begin{theorem}\label{casor=0} Given any $m,n\in\N$ such that $r\equiv 0\mod p$ and $p^{m-1}r< n\leq p^m r$,
the exponent of $H_r$ is exactly $p^m$.
\end{theorem}

\begin{proof} This is a consequence of Theorem \ref{kr} and Proposition \ref{sub1} (applied with $a=1$).
\end{proof}

\begin{theorem}\label{casor=-1} Suppose $p$ is odd. Given any $m,n\in\N$ such that $r\equiv -1\mod p$ and 
$$p^{m-1}(r+1)-1< n\leq p^{m}(r+1)-1,$$
the exponent of $H_r$ is exactly $p^m$.
\end{theorem}

\begin{proof} This follows from Theorem \ref{kr} and Proposition \ref{sub2}.
\end{proof}

\begin{theorem}\label{caso22} Suppose $p=2$. Given any $m,n\in\N$ such that $r\equiv 1\mod 2$, $r>1$, and 
$$2^{m-1}(r+1)-1< n\leq 2^{m}(r+1)-1,$$
the exponent of $H_r$ is exactly $2^m$.
\end{theorem}

\begin{proof} Use Theorem \ref{kr} and Proposition \ref{sub5} (applied repeatedly, starting with $a=b=1$).
\end{proof}

We next compute the exponent of $H_1=G_n$ itself when $p=2$. 

\begin{theorem}\label{caso222} Suppose $p=2$. Given any $m,n\in\N$ such that $n<2^{m+1}$, 
the exponent of $G_n$ is at most $2^m$, and it is equal to $2^m$ if $2^m<n<2^{m+1}$ or $2^m=n$
is equal to 2 or 4.
\end{theorem}

\begin{proof} The first statement is a special case of Theorem \ref{kr}. If $2^m<n<2^{m+1}$, then
Theorem~\ref{kr} and Proposition \ref{sub1} applied with $a=1$ show that the exponent of $G_n$ is $2^m$.
%By Theorem~\ref{kr}, the exponent of $G_4$ is at most 4.
Since $G_2\cong R^+$, its exponent is 2. Moreover, as $x+x^2$ has order 4 in $G_4$, the exponent of $G_4$ is 4. 
\end{proof}

It remains to determine the exponent of $G_{2^m}$ when $p=2$ and $m>2$. 

\begin{prop}\label{pro2} Suppose $p=2$ and $m>2$. Then the exponent of $G_{2^m}$ is at most $2^m$
and at least $2^{m-1}$. More precisely, the exponent {\color{red} of} $G_{2^m}$ is equal to $2^{m-1}$ if
$ab^{2^{m-2}-1}(a^2+b)^{2^{m-2}}=0$ for all $a,b\in R$, and equal to $2^m$ otherwise.
\end{prop}

\begin{proof} The first statement follows from Theorem \ref{kr} and Proposition \ref{sub1} (applied with $a=1$).
As for the second statement, a repeated application of Proposition \ref{sub5}, starting with $r=1$, yields
$$
{\color{red}
f^{(2^{m-1})}}=x+a b^{2^{m-2}-1}(a^2+b)^{2^{m-2}}x^{2^{m}}+b^{2^{m-2}} (a^2+b)^{2^{m-2}}x^{2^{m}+1}+\cdots.
$$
Now the first statement gives the desired result.
\end{proof}

%\begin{cor}\label{c2} If $p=2$, $m>2$, and $R$ is a Boolean ring, then the exponent
%of $G_{2^m}(R)$ is $2^{m-1}$.
%\end{cor}

\begin{theorem}\label{caso2} Suppose $p=2$. If $R$ is a Boolean ring
the exponent of $G_{2^m}$ is $2^{m-1}$ for every $m>2$.
If $R$ is not a Boolean ring, the exponent of $G_{2^m}$ is $2^{m}$ for every $m>2$.
\end{theorem}

\begin{proof} If $R$ is a Boolean ring, then the exponent of $G_{2^m}$ is $2^{m-1}$ for every $m>2$ by Proposition~\ref{pro2}.
Suppose next that the exponent of $G_{2^m}$ is $2^{m-1}$ for some $m>2$. Then, by Proposition~\ref{pro2} applied with $a=1$, $R$
is algebraic over $\Z_2$ and the minimal polynomial over $\Z_2$ of every element of $R$ has no irreducible
factors outside of $x$ and $x+1$. Let $b\in R$ be arbitrary. We claim that $b^{2^{m-2}}=b^{2^{m-2}-1}$. 
Indeed, setting $a=b+1$, then $a^2+b=b^2+b+1$
is invertible in $R$, and from $ab^{2^{m-2}-1}(a^2+b)^{2^{m-2}}=0$, the claim follows.
If $m=3$ we are done. Suppose next $m>3$. 
By the above, $(b+1)^{2^{m-2}}=(b+1)^{2^{m-2}-1}$. Since every binomial number ${\binom{2^n-1}{i}}$, $n\geq 1$, is odd, we deduce
$$
b^{2^{m-2}}+1=b^{2^{m-2}-1}+b^{2^{m-2}-2}+\cdots+b+1,
$$
that is
$$
0=b^{2^{m-2}-2}+\cdots+b=b(b+1)f(b),
$$
where $f(x)\in \Z_2[x]$ is relatively prime to $x(x+1)$. Thus $f(b)$ is invertible in $R$, hence $b^2=b$.
\end{proof}

%For instance, if $R=\Z_2[x]$, or $R=\Z_2[x]/(f)$, where $0\neq f\in\Z_2[x]$ has an irreducible factor different from $x$ and $x+1$,
%or $R=\Z_2[x]/(f)$, where $f=x^{2^{m-2}}(1+x)^{2^{m-2}}$, with $m>2$, then the exponent of $G_{2^m}(R)$ is $2^m$. Also, if
%$R=\Z_2[x_1,x_2,\dots]/I$, where $I=(x_1(x_1+1),x_2^2(x_2+1)^2,\dots)$, then the exponent of $G_{2^m}(R)$ is $2^m$
%for any $m>2$. On the other hand, if $R$ is the Boolean ring, then (\ref{b}) shows that the exponent
%of $G_{2^m}(R)$ is $2^{m-1}$ for any $m>2$.

\section{The exponent of $H_r$ when $r\not\equiv 0,-1\mod p$}\label{s5}

\begin{prop}\label{poli} Let $\ell\in\N$ satisfy $\ell<p-1$, take $c\in R$, and consider the polynomials 
$T_0,\ldots,T_\ell,Q\in R[x]$ defined as follows:
$$
T_0=x2x3x\cdots \ell x,\; T_1=(x-c)2x3x\cdots \ell x,\; T_2=(x-c)(2x-c)3x\cdots \ell x,\;\dots,
$$
$$
T_\ell=(x-c)(2x-c)\cdots (\ell x-c),\; Q=T_0+\cdots+T_\ell.
$$ 
Then
$$
Q=(\ell+1)!(x-c/(\ell+1))\cdots (x-c/2).
$$
\end{prop}

\begin{proof} The result is clear if $\ell=1$. Suppose $1<\ell<p-1$ and the result is true for $\ell-1$.
Then
$$
\begin{aligned}
Q & =\ell! (x-c/\ell)\cdots (x-c/2)\ell x + (x-c)(2x-c)\cdots (\ell x-c)\\
&=\ell !(x-c/\ell)\cdots (x-c/2) (\ell x+(x-c))
=(\ell+1)!(x-c/(\ell+1))\cdots (x-c/2).
\end{aligned}
$$
\end{proof}

\begin{prop}\label{improve} Suppose that $0<t<p-1$ and
let
$
f=x+a x^{r+1}+b x^{r+t+1}+\cdots.
$ 
Then there are $A,B\in R$ such that
$
f^{(p)}=x+A x^{pr+t+1}+B x^{pr+2t+1}+\cdots.
$

Moreover, set $k\in\{1,\dots,p\}$ such that $kr\equiv -1\mod p$, as well as
$e=pr+2t+1$, and let $j\in J(e)$. Then $j$ is obtained from $j_0=1$ by means
of $p-2$ jumps of length $r$ and 2 jumps of length $r+t$, or $p-1$ jumps of length $r$ and 1 jump of length $r+2t$.
In the latter case, either $j_{k-1}=(k-1)r+1$ and $j_k=kr+2t+1$, or else $j_{k-2}=(k-2)r+1$ and 
$j_{k-1}=(k-1)r+2t+1$, except that if $p=3$ and $r=t$, then $j=(1,r+1,2r+1,5r+1)$ is also possible.
In the former case, either there is some $k<i\leq p$ such that
\begin{equation}
\label{du0}
j_{k-1}=(k-1)r+1,j_{k}=kr+t+1,\dots,j_{i-1}=(i-1)r+t+1,j_{i}=ir+2t+1,\dots,j_p=pr+2t+1,
\end{equation}
or there is some $1\leq i<k-1$ such that 
\begin{equation}
\label{du}
j_{i-1}=(i-1)r+1, j_{i}=ir+t+1,\dots,j_{k-2}=(k-2)r+t+1, j_{k-1}=(k-1)r+2t+1,\dots, j_p=pr+2t+1.
\end{equation}
\end{prop}

\begin{proof} We begin by showing that $f^{(p)}$ has the stated shape. 
By Proposition \ref{sub3}, it suffices to prove that given any $1\leq i<t$, we have $\Delta^p_{1,d}=0$ for 
$d=pr+t+i+1$. By assumption, $r\not\equiv -1\mod p$, so $1<k\leq p-1$.  

Consider the sequences $j^0,\dots,j^{p-1}\in L(d)$, all of which start at 1, end at $d$,
consist of $p-1$ jumps of length $r$ and a single jump of length $r+t+i$, and are defined as follows for any $0\leq u\leq p-1$:
$$
j^u=(j^u_0,\dots,{\color{red}j^u_{p}})=(1,r+1,\dots,ur+1, (u+1)r+t+i+1,\dots, pr+t+i+1),
$$
so that the single jump of length $r+t+i$ occurs when $j^u_{u+1}-j^u_u=r+t+i$.

\medskip

\noindent{\sc Claim 1.} $J(d)\subseteq\{j^0,\dots,j^{p-1}\}$. Indeed, let $j\in J(d)$. If there are $p$ jumps of length $r$, 
then $j_p=pr+1<d$, a contradiction. If there are $s$ jumps of length $r$ for some $0\leq s\leq p-2$, then by~(\ref{dirt}),
$j_p\geq sr+(p-s)(r+t)+1=pr+(p-s)t+1\geq pr+2t+1>d$, which is absurd.

\medskip

Let $g$ stand for the coefficient of $x^{r+t+i+1}$ in $f$.

\medskip

\noindent{\sc Claim 2.} We have $\Delta_{u,v}=au$ if $v-u=r$ and $\Delta_{u,v}=g u$ if $v-u=r+t+i$. Indeed,
if $v-u=r$, then the only way to obtain $x^v$ in $f^u$ is by selecting $x^{r+1}$ once and $x$ the remaining $u-1$ times.
If $v-u=r+t+i$, then $x^v$ can be produced in $f^u$ by choosing $x^{r+t+i+1}$ once and $x$ the remaining $u-1$ times.
A selection of $x$ of $u-3$ times or less will not produce $x^v$ a single time, as $3r$ is already larger than $r+t+i$.
Choosing $x$ exactly $u-2$ times also fails to produce $x^v$. This is so because $2r+t>r+t+i$,
so one would be forced to select $x^{r+1}$ twice, but in this case $2r=r+t+i$ leads to $r=t+i$, and therefore $p\mid i$,
against the fact that $1\leq i<t<p-1$.

\medskip

\noindent{\sc Claim 3.} If $j\in J(d)$, then $S(j)=a^{p-1}g j_1\cdots j_{p-1}$. This follows from Claims 1 and 2.

\medskip

\noindent{\sc Claim 4.} $J(d)\subseteq\{j^0,\dots,j^{k-1}\}$. This is an immediate consequence of Claims 1 through 3,
as any $j^u$, $k\leq u\leq p-1$, satisfies $j^u_k=kr+1$, which is congruent to 0 modulo $p$.

\medskip

We next set $q\in\{1,\dots,p\}$ such that  $q\equiv k(1+i)\mod p$, noting that $rq\equiv -(1+i)\mod p$.
We also set $c=t+i$.

\medskip

\noindent{\sc Claim 5.} $J(d)\subseteq\{j^{q-1},\dots,j^{p-1}\}$. Indeed, if $0\leq u\leq q-2$,
then $j^u_{q-1}\equiv (q-1)r+t+i+1\equiv 0\mod p$, so $S(j)=0$ by Claim 3.

\medskip

It follows from Claims 4 and 5 that $\Delta^p_{1,d}=0$ if $k<q$, so we assume for the remainder of the
proof that $q\leq k$. Observe that for any $j\in \{j^{q-1},\dots,j^{k-1}\}$, the values
$$j_0=1,j_1=r+1,\dots,j_{q-1}=(q-1)r+1, j_{k+1}=(k+1)r+t+i+1,\dots j_{p-1}=(p-1)r+t+i+1$$
are independent of $j$. Setting $W(j)=j_q\cdots j_k$ for any $j\in L(d)$, we deduce from
Claims 3 through~5 that it suffices to show that $W(j^{q-1})+\cdots+W(j^{k})=0$, {\color{red} where
for convenience we have added the term $W(j^{k})=0$.} We next apply Proposition \ref{poli} with $\ell=k-q+1$ and $c=t+i$. 
{\color{red} This is a valid application as $1\leq i<t$, so $t>1$, hence $k\leq p-2$ and $\ell\leq k\leq p-2$.
Following the notation of Proposition \ref{poli}} 
and making use
of $qr+t+i+1\equiv r\mod p$, we have
$$
W(j^{q-1})=T_0(r), W(j^{q})=T_1(r),\dots, W(j^{k-1})=T_{k-q}(r), W(j^{k})=T_{k-q+1}(r),
$$
so
$$
W(j^{q-1})+\cdots+W(j^k)=Q(r)=(k-q+2)! (r-c/(k-q+2))\cdots (r-c/2)=0,
$$
since $(k-q+1)r\equiv c\mod p$ and $k-q+1\geq 2$, for $k=q$ leads to $i\equiv 0\mod p$,
which cannot be, as $1\leq i<t<p-1$.

Set $e=pr+2t+1$ and let $j\in J(e)$. The number of jumps involved in $j$ follows from the analysis made in Claim 1
(it is yet not clear when these jumps must occur).

Note that $e\equiv 0\mod p\Leftrightarrow k=2$. 

Assume next $k=2$. We claim that the first 2  
jumps cannot be both of length~$r$, unless $p=3$ and $r=t$. 
Otherwise, $j_2=2r+1$ is congruent to 0 modulo $p$; if $p>3$, then none of
$r,r+t,r+2t$ are congruent to 0 modulo $p$, so
Lemma \ref{York} and the foregoing description of $j$ disallow $j_2=2r+1$; if $p=3$, then 
$j=(1,r+1,2r+1, 3r+2t+1)$, and when $r>t$, we see that $\Delta_{j_2,j_3}=0$.
This proves the claim. Also, the first jump cannot be of length $r+t$,
for otherwise $j_1=r+t+1$ is congruent to 0 modulo $p$, and since $r,r+t\not\equiv 0\mod p$, Lemma \ref{York}
implies {\color{red} $j\notin J(e)$}, a contradiction. Thus, if $j$ consists of $p-2$ jumps of length $r$ and 2 jumps of length $r+t$,
then $j_1=r+1$, $j_2=2r+t+1$ as in (\ref{du0}), while (\ref{du}) never occurs.
Moreover, if $j$ consists of  $p-1$ jumps of length $r$ and 1 jump of length $r+2t$,
then either $j_1=r+1$ and $j_2=2r+2t+1$, or $j_1=r+2t+1$ and $j_2=2r+2t+1$.
This completes the proof when $k=2$.

Suppose finally that $k\neq 2$. Since $kr+1\equiv 0\mod p$, the first $k$ jumps cannot
be of length~$r$ by Lemma \ref{York}. Assume first that $j$ consists of  $p-1$ jumps of length $r$ and 1 jump of length $r+2t$.
If $j_i=ir+2t+1$ for some $1\leq i\leq k-2$, then {\color{red} $j_{k-2}=(k-2)r+2t+1$} is congruent to 0 modulo $p$, which is impossible.
Thus either $j_{k-1}=(k-1)r+1$ and $j_k=kr+2t+1$, or $j_{k-2}=(k-2)r+1$ and $j_{k-1}=(k-1)r+2t+1$.
Assume next that $j$ consists of $p-2$ jumps of length $r$ and 2 jumps of length $r+t$.
If $j_{k-1}=(k-1)r+1$, then necessarily $j_{k}=kr+t+1$, which is congruent to $t$ modulo $p$,
and the second jump of length $r+t$ may occur at any point afterward as in (\ref{du0}). Suppose, if possible, that there
is some $s<k-1$ such that $j_s=sr+2t+1$. Then $j_{k-2}=(k-2)r+2t+1$ is congruent to 0 modulo $p$, which cannot be. 
Thus, if $j_{k-1}\neq (k-1)r+1$, there is some $1\leq i<k-1$ such (\ref{du}) holds. This completes the proof.
\end{proof}

Proposition \ref{improve} is no longer valid as stated if we remove the hypothesis $0<t<p-1$ and replace it by $r\not\equiv 0,-1\mod p$.
Indeed, in this case $t=pu+t_0$, where $0<t_0<p-1$ and $u\geq 0$. Suppose that $u\geq 1$. Then the coefficients
of $x^{pr+t+1+pm}$ in $f^{(p)}$, $1\leq m\leq u$, need not be 0,
as it occurs in Proposition \ref{sub2}. This is revealed twice during the proof of Proposition \ref{improve}.
In Claim 2, $r=t+i$ is now possible, in which case if $v-u=r+t+i$ then $\Delta_{u,v}=gu+a^2\binom{u}{2}$.
Thus, if $r=t+i$, then for $j=j^u$ in Claim 2, we have $S(j)=(g j_u+a^2 \binom{j_u}{2})a^{p-1}\underset{v\neq u}\Pi j_v$, where
$v$ runs between 1 and $p-1$. These changes, occurring if $r=t+i$, do not affect the validity of Claims 4 and 5. However, if
$p|i$, that is, $i=pm$ for $1\leq m\leq u$ (which is implied by $r=t+i$ but not equivalent to it), then $q=k$, so
$J(d)=\{j\}$, {\color{red} $j=j^{k-1}$}, with $S(j)=a^{p-1}g j_1\cdots j_{p-1}$ if {\color{red}$r\neq t+i$} and 
$S(j)=(g j_k+a^2 \binom{j_k}{2})a^{p-1}\underset{v\neq k}\Pi j_v$ if {\color{red}$r=t+i$}. Since no $j_u$ is congruent to 0 modulo $p$,
suitable choices of $a$ and $g$ will make $\Delta^p_{1,d}\neq 0$.

%\begin{prop}\label{nuevap} {\color{red} Suppose that $p=3$, $r=t$, $r\equiv 1\mod 3$, and let
%$
%f=x+a x^{r+1}+b x^{2r+1}+\cdots.
%$ 
%Then there are $A,B\in R$ such that
%$
%f^{(p)}=x+A x^{4r+1}+B x^{5r+1}\cdots.
%$

%Moreover, set $e=pr+2r+1$, and let $j\in J(e)$. Then $j$ is one of the following four sequences:
%$$
%(1,r+1,2r+1,5r+1), (1,r+1,4r+1,5r+1), (1,3r+1,4r+1,5r+1), (1,r+1,3r+1,5r+1),
%$$
%where the first one is not accounted for in Proposition \ref{improve} and all three others are.
%}
%\end{prop}

\begin{theorem}\label{improve2} Suppose that $0<t<p-1$  and let
$
f=x+a x^{r+1}+b x^{r+t+1}+\cdots.
$
Then
$$
f^{(p)}=x+A x^{pr+t+1}+B x^{pr+2t+1}+\cdots,
$$
where if $r>t$,
$$
A=-a^{p-1}b\quad\text{ and }\quad B=-a^{p-2}b^2,
$$
while if $r=t$, 
$$
A=-a^{p-1}b+\frac{r+1}{2}a^{p+1}\quad\text{ and }\quad B=-a^{p-2}b^2+{\color{red}(r+1)a^{p}b}-\frac{(r+1)^2}{4}a^{p+2}.
$$
\end{theorem}

\begin{proof} That $f^{(p)}$ has the stated shape follows from Proposition \ref{improve}.
Moreover, the value of $A$ is given in Proposition \ref{sub3}. It remains confirm the value of $B$.
Set $k\in\{1,\dots,p\}$ such that $kr\equiv -1\mod p$ {\color{red}and let $c$ be the coefficient of $x^{r+2t+1}$ appearing in $f$.}
We exclude the case $p=3$ and $r=t$ until further notice.

Let $d=pr+2t+1$. We write $j^1$ and $j^2$ for the 2 sequences consisting of $p-1$ jumps of length $r$ and 1 jump of length $r+2t$
described in Proposition \ref{improve}, in the order listed therein. We proceed to compute $S(j^1)+S(j^2)$.

Let $j\in\{j^1,j^2\}$. If $r>t$, then $r=pm+t$, with $m\geq 1$, and we see that for any $0\leq i<p$,
the coefficient of $x^{j_{i+1}}$ in $f^{j_i}$ can be obtained in one and
only one way. However, if $r=t$, precisely when $j=j^1$ and $i=k-1$, or when $j=j^2$ and $i=k-2$, and not otherwise,
this coefficient can be achieved in three different ways, namely using $x^{r+1}$ thrice, or $x^{r+1}$ and $x^{2r+1}$ once each, or 
$x^{3r+1}$ once; in either of these three ways, the remaining choices are of $x$ itself.
Thus if $r>t$, we have
$$
\Delta_{j^1_{i}, j^1_{i+1}}=\begin{cases} (ir+1)a & \text{ if }0\leq i<k-1,\\ ((k-1)r+1)c & \text{ if }i=k-1,\\
(ir+2t+1)a & \text{ if }k-1<i<p,\end{cases} \Delta_{j^2_{i}, j^2_{i+1}}=\begin{cases} (ir+1)a & \text{ if }0\leq i<k-2,\\ ((k-2)r+1)c & \text{ if }i=k-2,\\
(ir+2t+1)a & \text{ if }k-2<i<p,\end{cases} 
$$
while if $r=t$, then
$$
\Delta_{j^1_{i}, j^1_{i+1}}=\begin{cases} (ir+1)a & \text{ if }0\leq i<k-1,\\ ((k-1)r+1)c+\binom{(k-1)r+1}{(k-1)r-1,1,1}ab+\binom{(k-1)r+1}{3}a^3 & \text{ if }i=k-1,\\
((i+2)r+1)a & \text{ if }k-1<i<p,\end{cases}
$$
and
$$
\Delta_{j^2_{i}, j^2_{i+1}}=\begin{cases} (ir+1)a & \text{ if }0\leq i<k-2,\\ ((k-2)r+1)c+\binom{(k-2)r+1}{(k-2)r-1,1,1}ab+
\binom{(k-2)r+1}{3}a^3 & \text{ if }i=k-2,\\
((i+2)r+1)a & \text{ if }k-2<i<p,\end{cases}
$$
where if $k=2$ the last 2 multinomial numbers are understood to be 0.

Suppose first that $r>t$, and set $g=ca^{p-1}g_1g_2$, where $g_1$ is the product of all $ir+1$ with $0\leq i\leq k-2$,
and $g_2$ is the product of all $ir+2t+1$ with $k-1<i<p$. Then by the above $S(j^1)+S(j^2)=g((k-1)r+1+(k-1)r+2t+1)=0$.

Suppose next that $r=t$, and set $h=a^{p-1}h_1h_2$, where $h_1$ is the product of all $ir+1$ with $0\leq i<k-2$,
and $h_2$ is the product of all $(i+2)r+1$ with $k-1<i<p$. We also set
$$u_1=((k-2)r+1)\left(((k-1)r+1)c+\binom{(k-1)r+1}{(k-1)r-1,1,1}ab+\binom{(k-1)r+1}{3}a^3\right),
$$
$$u_2=((k+1)r+1)\left(((k-2)r+1)c+\binom{(k-2)r+1}{(k-2)r-1,1,1}ab+\binom{(k-2)r+1}{3}a^3\right),$$
where, again, if $k=2$ the last 2 multinomial numbers are meant to be equal to 0.

Then by above $S(j^1)+S(j^2)=h(u_1+u_2)$. By Wilson's theorem, the product of all $ir+1$, 
with $0\leq i<p$ and $i\neq k$, is equal to $-1$, and by definition, $h_1h_2$ is obtained from
this product by repeating the factors $1$ (when $i=p-2$) and $r+1$ (when $i=p-1$) and omitting
the factors $(k-2)r+1$, $(k-1)r+1$, and $(k+1)r+1$ (which multiply to $2r^3$). Thus $h=-a^{p-1}(r+1)/2r^3$.
{\color{red} Since $kr\equiv -1\mod p$, we have $(k-2)r+1\equiv -2r, (k-1)r+1\equiv -r, (k+1)r+1\equiv r\mod p$.
Moreover, the multinomial coefficient $\binom{N}{N-2,1,1}$ equals $N(N-1)$. Therefore, if $k\neq 2$, then
$$
u_1=-2r(-rc+r(r+1)ab-\frac{r(r+1)(r+2)}{6}a^3),\; u_2=r(-2rc+2r(2r+1)ab-\frac{2r(2r+1)(2r+2)}{6}a^3),
$$
so that $u_1+u_2=2r^3ab-r^3(r+1)a^3$.  If $k=2$ (recalling that the case $p=3$ is still excluded), 
then $u_1$ remains as displayed above
but $u_2=-2r^2c$. As $r=-1/2$ in the prime field of $R$, then $u_1+u_2=-\frac{ab}{4}+\frac{a^3}{16}$.
In both cases, $u_1+u_2=2r^3ab-r^3(r+1)a^{3}$.} Therefore, 
\begin{equation}
\label{B1}
S(j^1)+S(j^2)=(-\frac{a^{p-1}(r+1)}{2r^3})({\color{red}2}r^3ab-r^3(r+1)a^{3})={\color{red}-(r+1)a^pb}+
\frac{(r+1)^2}{2}a^{p+2}.
\end{equation}

We next compute that sum of all $S(j)$, where $j$ runs through all $p-2$ sequences described in (\ref{du0})~and~(\ref{du}).
For this purpose, we abandon the above meaning for $j^1$ and $j^2$, and let $j^i$, with $k<i\leq p$ or $1\leq i<k-1$, stand for the 
$p-2$ sequences described in (\ref{du0})~and~(\ref{du}), 
setting $S(i)=S(j^i)$. By the above, if $r>t$, then $B$ is the sum of the $S(i)$, while if $r=t$, then $B$ is the sum of the 
$S(i)$ and (\ref{B1}).

Suppose first that $r>t$. Then by Wilson's theorem, 
$$
S(p)=a^{p-2}b^2 (r+1)(2r+1)\cdots ((k-1)r+1)((k+1)r+1)\cdots (pr+1)=-a^{p-2}b^2.
$$
Moreover, 
$$
S(p-1)=S(p)(r+1)/(pr+1),\dots, S(k+1)=S(p)(r+1)/((k+2)r+1),
$$
$$
S(k-2)=S(p)(r+1)/((k-2)r+1),\dots,S(1)=S(p)(r+1)/(r+1).
$$
Thus
$$
B=S(p)+S(p)(r+1)\big(\frac{1}{pr+1}+\cdots+\frac{1}{(k+2)r+1}+\frac{1}{(k-2)r+1}+\cdots+\frac{1}{r+1}\big).
$$
The parenthetical expression is the sum of all units in $\Z_p$, namely 0, minus 
$\frac{1}{(k+1)r+1}+\frac{1}{(k-1)r+1}=\frac{1}{r}+\frac{1}{-r}=0$. We conclude that $B=S(p)=-a^{p-2}b^2$ when $r>t$.

Suppose next that $r=t$. Referring to (\ref{du0}), note that $x^{j^i_{i}}$ in $f^{j^i_{i-1}}$ and $x^{j^i_k}$ in $f^{j^i_{k-1}}$ can be obtained using $x^{2r+1}$ once or $x^{r+1}$ twice. Thus
$$
S(p)=a^{p-2} (r+1)(2r+1)\cdots (((k-1)r+1)b+\binom{(k-1)r+1}{2}a^2)((k+1)r+1)\cdots ((pr+1)b+\binom{pr+1}{2}a^2).
$$

As $pr+1\equiv 1\mod p$ the last factor reduces to $b$. Taking the common factor $-r$ in the $(k-1)$th term we obtain 
$S(p)=-a^{p-2}b^2+\frac{r+1}{2}a^pb$, by Wilson's theorem. Moreover, setting $S'(p)=-a^{p-2}b+\frac{r+1}{2}a^p$, we have
$$
S(p-1)=S'(p)(r+1)\left(\frac{b}{pr+1}+\frac{(p-1)ra^2}{2(pr+1)}\right), \dots, 
$$
$$
S(k+1)=S'(p)(r+1)\left(\frac{b}{(k+2)r+1}+\frac{(k+1)ra^2}{2((k+2)r+1)}\right),
$$

$$
S(k-2)=S'(p)(r+1)\left(\frac{b}{(k-2)r+1}+\frac{(k-3)ra^2}{2((k-2)r+1)}\right),\dots, 
S(1)=S'(p)(r+1)\frac{b}{r+1}.
$$
Thus, letting $T=\{1\leq i\leq p\,|\, i\neq k, i\neq k-1\}$ and $U=T\setminus \{k+1\}$, we see that
$$
\sum_{i\in T} S(i)=S(p)+S'(p)(r+1)\left(\sum_{i\in U} \frac{b}{ir+1}+\sum_{i\in U}\frac{(i-1) ra^2}{2(ir+1)}\right).
$$
Here the same argument used in the case $r>t$ yields
$$
\sum_{i\in U} \frac{b}{ir+1}=0.
$$
Regarding the second sum, since $r\not\equiv -1\mod p$, the function $\varphi:\Z_p\setminus\{k\}\to \Z_p\setminus\{1/2\}$, given by
$$
\varphi(i)=\frac{(i-1)r}{2(ir+1)},
$$
is a bijection. Since the sum of all elements of $\Z_p$ is 0, we infer that
$$
\sum_{i\in U}\frac{(i-1) ra^2}{2(ir+1)}=a^2\left(-\frac{1}{2}-\frac{kr}{2((k+1)r+1)}-\frac{(k-2)r}{2((k-1)r+1)}\right)=-3a^2/2.
$$
Hence
\begin{equation}
\label{B2}
\sum_{i\in T} S(i)=
-a^{p-2}b^2+2(r+1)a^{p}b-\frac{3(r+1)^2}{4}a^{p+2}.
\end{equation}
Addition of (\ref{B1}) with (\ref{B2}) yields the stated expression of $B$ when $r=t$. This completes the proof
of all cases but $p=3$ and $r=t$ 

If $p=3$ and $r=t$, then {\color{red}$r=t=1$}, and any $j\in J(d)$ is one of the following:
{\color{red}
$$
(1,2,3,6), (1,2,5,6), (1,4,5,6), (1,2,4,6),
$$
and the corresponding values $S(j)$ are respectively equal to
$$
-a^5, a^2c+a^3b, -a^2c, -ab^2+a^3b,
$$
and therefore $B=-ab^2-a^3b-a^5$}, as claimed. This completes the proof.
\end{proof}

\begin{cor}\label{formu} Suppose that {\color{red}$0<t<p-1$} and take any $m\in\N$.
Then if $r=t$ and $f\equiv x+x^{r+1}\mod x^{2r+2}$, we have
$$
f^{(p^m)}=x+((r+1)/2)^m x^{p^m r+p^{m-1}r+\cdots+pr+r+1}-((r+1)/2)^{m+1} x^{p^m r+p^{m-1}r+\cdots+pr+2r+1}+\dots,
$$
while if $r>t$ and $f=x+x^{r+1}+x^{r+t+1}+\cdots$, we have
$$
f^{(p^m)}=x+(-1)^m x^{p^m r+p^{m-1}t+\cdots+pt+t+1}+(-1)^{m} x^{p^m r+p^{m-1}t+\cdots+pt+2t+1}+\dots.
$$
\end{cor}

\begin{proof} The first formula follows from the case $r=t$ of Theorem \ref{improve2} applied with $a=1$ and $b=0$,
followed by a repeated application of the case $r>t$ of Theorem \ref{improve2}.  
The second formula follows from a repeated application of the case $r>t$ of
Theorem \ref{improve2} starting with $a=b=1$.
\end{proof}

\begin{cor}\label{lip}
\label{Prop g_i^p char odd}
Suppose $p\neq 2$ and set $g_i=x+x^i$, where $1<i\leq p$. Then
$$
g_i^{(p)}\equiv x+\frac{i}{2}x^{p(i-1)+i}-\frac{i^2}{4}x^{p(i-1)+2i-1}\mod x^{(p+3)(i-1)+1}.
$$
\end{cor}

\begin{proof} Set $r=i-1$, so that $1\leq r=t<p$. The binomial expansion of $g_{r+1}^m$, $m\geq 1$, makes it clear that for any $d>1$,
if $j\in J(d)$, then every jump occurring in $j$ must be a multiple of $r$. 

If $i=p$, Proposition \ref{sub2} applied with $a=1$ and $b=0$ yields $g_p^{(p)}\equiv x\mod x^{(p+2)(p-1)+2}$, so
$g_p^{(p)}\equiv x\mod x^{(p+3)(p-1)+1}$ by our initial remark (Example \ref{ejemplo} below computes $g_p^{(p)}$ exactly).

If $r<p-1$, Corollary \ref{formu} gives 
$g_{i}^{(p)}\equiv x+\frac{i}{2}x^{p(i-1)+i}-\frac{i^2}{4}x^{p(i-1)+2i-1}\mod x^{(p+2)(i-1)+2}$, 
so $g_{i}^{(p)}\equiv x+\frac{i}{2}x^{p(i-1)+i}-\frac{i^2}{4}x^{p(i-1)+2i-1}\mod x^{(p+3)(i-1)+1}$ by our initial remark.
\end{proof}

\begin{theorem}\label{casorgeneral} Suppose that $p$ is odd, that $0\leq t<p$,
and that for some $m,n\in\N$, we have
$$p^{m-1} r+p^{m-2}t+\cdots+pt+t<n\leq p^m r+p^{m-1}t+\cdots+pt+t.$$ Then the exponent of $H_r$ is exactly $p^m$.
\end{theorem}

\begin{proof} This follows from Theorem \ref{casor=0} if $t=0$, Theorem \ref{casor=-1} if $t=p-1$,
and Theorem \ref{kr} together with Corollary \ref{formu} {\color{red} if $0<t<p-1$.}
\end{proof}

%\begin{note} If we manage to get a shape for $\alpha$ that repeats itself after rising to the $p$th power
%in Propositions \ref{improve2} and Proposition \ref{improve3}, then Proposition \ref{improve} can be eliminated
%and we do not need to ask that $t|p-1$. Then Corollary \ref{formu} can be stated without this hypothesis.
%The fact that all coefficients are 0 between $x^{r+t+1}$ and $x^{r+2t+1}$ 
%does not repeat itself without additional assumptions. 
%\end{note}

%\section{Normal subgroups of $G_n(R)$}

Let $T$ be a finite $p$-group. Then $T$ is said to be regular if for all $s,t\in T$ and any integer $i>0$,
we have $(st)^{p^i}\equiv s^{p^i} t^{p^i}\mod [\langle s,t\rangle,\langle s,t\rangle]^{p^i}$. Moreover, if $p$ is odd,
then $T$ is said to be powerful if $[T,T]\subseteq T^p$.

\begin{exa}\label{ejemplo} Suppose $p$ is odd and that $R$ has $p^u$ elements for some $u\in\N$.
Then $H_1,\dots, H_{p-1}$ are not regular subgroups of $G_{p^m}$ for any $m\geq 2$. Indeed, let $f=x+x^p$.
If $d>1$ and $j\in J(d)$, then necessarily $j_1=p$, which forces $j_2=p^2,j_3=p^3,\cdots,d=j_p=p^p$,
yielding $f^{(p)}=x+x^{p^p}$ (exactly). Repeating this process, we find that
$f^{(p^{m-1})}=x+x^v$ (exactly),
where $v=p^{p^{m-1}}$. Since $p^{m-1}>m$, we infer $v>p^m$, which implies 
$f^{(p^{m-1})}\in K_{p^m}$. Next let $g=x+x^{2p-1}+\cdots$. Then
$g\in K_p$, so Theorem \ref{je} yields  $g^{(p^{m-1})}\in K_{p^m}$. Let $s\in [\langle f,g\rangle,
\langle f,g\rangle]$. Then Proposition \ref{com1} ensures that $s\in K_p$, and therefore $s^{(p^{m-1})}\in K_{p^m}$
by Theorem \ref{je}. Thus, 
$$f^{(p^{m-1})}*g^{(p^{m-1})}*s^{(p^{m-1})}\in K_{p^m}.$$ 
On the other hand, 
$$f*g=x+x^p+x^{2p-1}+\cdots,\text{ so }(f*g)^{(p^{m-1})}=x+(-1)^{m-1}x^{p^m}+\cdots$$ 
by Proposition~\ref{sub2}, which ensures that $(f*g)^{(p^{m-1})}\notin K_{p^m}$. Thus $H_{p-1}$ is not regular,
and therefore $H_1,\dots, H_{p-1}$ are not regular.

It should be noted that $K_{p-1}^{(p^{m-1})}\subseteq K_{p^m-1}$ by repeated application of Theorem \ref{kr},
and that $(f*g)^{(p^{m-1})}$ is an element of $K_{p-1}^{(p^{m-1})}$ that is not in $K_{p^m}$. On the other hand,
every element of $h\in K_{p-1}$ is of the form $h=h_1 h_2$, where $h_1=x+ax^p$ and $h_2=x+b x^{p+1}+\cdots$ for suitable
$a,b\in R$, and these components satisfy $h_1^{(p^{m-1})},h_2^{(p^{m-1})}\in K_{p^m}$.

It is much easier to see that $H_1,\dots, H_{p-1}$ are not powerful subgroups of $G_{p^m}$ for any $m\geq 2$.
{\color{red} Indeed, suppose $1\leq i\leq p-1$. Then by Proposition \ref{com1}, we have $[x+x^{i+2},x+x^{i+1}]=x+x^{2i+2}+\cdots$,
and since $2i+2\leq pi+i<p^m$, the image of the commutator belongs to $[H_i,H_i]$ but does not lie in  $H_{pi+i}\supseteq H_i^p$.}
\end{exa}

\noindent{\bf Acknowledgments.} We are very grateful to the referees for a careful reading of the paper
and numerous valuable suggestions for improvement, and to A. Herman for his help with GAP computations.

%===================================================

\end{document}